\documentclass{amsart}

\usepackage{amssymb, amscd}
\usepackage[all]{xy}
\usepackage{graphicx}
\usepackage{hyperref}

\numberwithin{equation}{section}

\newtheorem{thm}{Theorem}[section]
\newtheorem{prop}[thm]{Proposition}
\newtheorem{lem}[thm]{Lemma}
\newtheorem{cor}[thm]{Corollary}

\theoremstyle{definition}

\theoremstyle{remark}
\newtheorem{rem}[thm]{Remark}
\newtheorem{exmp}[thm]{Example}

\renewcommand{\hom}{\operatorname{Hom}}

\newcommand{\Z}{\mathbb{Z}}
\newcommand{\Q}{\mathbb{Q}}
\newcommand{\R}{\mathbb{R}}
\newcommand{\C}{\mathbb{C}}
\newcommand{\F}{\mathbb{F}}

\DeclareMathOperator{\bl}{Bl}

\DeclareMathOperator{\ext}{Ext}

\DeclareMathOperator{\lk}{lk}

\DeclareMathOperator{\rank}{rank}

\DeclareMathOperator{\tor}{Tor}

\begin{document}

%%%%%%% Title %%%%%%%%%%%%%%%%%%%%%%%%%%%%%%%%%%%%%%%%%%%%%%%%%%%%%%%%%%%%%%%%%
\title[Blanchfield pairings and Gordian distance]{Blanchfield pairings and Gordian distance}

\author[S.~Friedl]{Stefan Friedl}
\address{Department of Mathematics, University of Regensburg, Germany}
\email{sfriedl@gmail.com}
\author[T.~Kitayama]{Takahiro Kitayama}
\address{Graduate School of Mathematical Sciences, the University of Tokyo, Japan}
\email{kitayama@ms.u-tokyo.ac.jp}
\author[M.~Suzuki]{Masaaki Suzuki}
\address{Department of Frontier Media Science, Meiji University, Japan}
\email{mackysuzuki@meiji.ac.jp}


\begin{abstract}
A lower bound of the Gordian distance is presented in terms of the Blanchfield pairing.
Our approach, in particular, allows us to show at least for $195$ pairs of unoriented nontrivial prime knots with up to $10$ crossings that their Gordian distance is equal to $3$, most of which are difficult to treat otherwise.
\end{abstract}

%\date{\today}

\sloppy

\maketitle

%%%%%%% Section 1 %%%%%%%%%%%%%%%%%%%%%%%%%%%%%%%%%%%%%%%%%%%%%%%%%%%%%%%%%%%%%
\section{Introduction}

In this paper we present a lower bound of the Gordian distance of knots in terms of their Blanchfield pairings, extending the work~\cite{BF14b, BF15} by Borodzik and the first author on the unknotting number of a knot.
We also describe our computational results on the lower bound for all pairs of nontrivial knots (possibly non-prime) with up to $10$ crossings. 

\subsection*{Blanchfield pairings and Gordian distance}

Let $J$ and $K$ be knots in $S^3$.
Throughout this paper all knots are understood to be oriented.
The \textit{Gordian distance between $J$ and $K$}, which we denote by $d(J, K)$, is defined to be the minimal number of crossing changes necessary to turn $J$ into $K$.
This induces a metric on the set of isotopy classes of knots.
The \textit{unknotting number} $u(J)$ of $J$ is defined to be the Gordian distance between $J$ and the unknot.
We refer the reader to \cite{BFP16, C19, Liv20} for recent studies on lower bounds on the Gordian distance of knots.

We set $\Lambda = \Z[t, t^{-1}]$, which is equipped with the involution given by $\overline{p(t)} = p(t^{-1})$ for $p(t) \in \Lambda$.
Let $R$ be a Noetherian unique factorization domain with (possibly trivial) involution $\bar{\cdot} \colon R \to R$.
We call a homomorphism $\varphi \colon \Lambda \to R$ \textit{admissible} if it satisfies the following:
\begin{enumerate}
\item $\varphi(t) \neq 1$ and $\varphi(t^{-1}) = \overline{\varphi(t)}$,
\item For every finitely generated $R$-module $L$ the module $\hom_R(L, R)$ is free.
\end{enumerate}
Every principle ideal domain satisfies (2), and so do $\Lambda$ and its localization.
(See \cite[Lemma 2.1]{BF14b} for the proof for $\Lambda$.
The same proof works also for its localization.)

Let $\varphi \colon \Lambda \to R$ be an admissible homomorphism.
We denote by $\Delta_J(t)$ the Alexander polynomial of a knot $J$.
If $\varphi(\Delta_J(t)) \neq 0$, then the Blanchfield pairing of $J$ associated with $\varphi$ is defined:
\[ \bl_{J, \varphi} \colon H_1^\varphi(X_J; R) \times H_1^\varphi(X_J; R) \to Q(R) / R, \]
where $X_J$ is the complement of an open tubular neighborhood of $J$ in $S^3$ and $Q(R)$ is the quotient ring of $R$.
(See Section~\ref{subsec_blanchfield}.)
When $\varphi$ is the identity map on $\Lambda$, $\bl_{J, \varphi}$ is the ordinary Blanchfield pairing of $J$.
For a hermitian matrix $A$ of size $n$ over $R$ with $\det A \neq 0$ we denote by $\lambda(A)$ the pairing
\[ \lambda(A) \colon R^n / A R^n \times R^n / A R^n \to Q(R) / R,~ (v, w) \mapsto \bar{v}^T A^{-1} w, \]
where we view $v$ and $w$ as represented by column vectors in $R^n$.
We write $A_0$ for the image of $A$ under the quotient map $R \to R / (\varphi(t) - 1) R$.
We define $n_\varphi(J)$ to be the minimal size of a hermitian matrix $A$ over $R$ such that
\begin{enumerate}
\item $\lambda(A)$ is isometric to $\bl_{J, \varphi}$,
\item the matrix $A_0$ is congruent over $R / (\varphi(t) - 1) R$ to a diagonal matrix which has $\pm 1$ on the diagonal.
\end{enumerate}
We will show that such a hermitian matrix $A$ exists for every admissible homomorphism $\varphi$ and every knot $J$.
(See Theorem~\ref{thm_main2}.)
When $\varphi \colon \Lambda \to S[t^{\pm 1}]$ is the embedding induced by a subring $S$ of $\C$, we write $n_S(J)$ for $n_\varphi(J)$.
In particular, $n_{\Z}(J)$ coincides with the invariant $n(J)$ first introduced in \cite{BF15}.

\subsection*{The Main theorem}

The following is the main theorem of this paper.
We write $-J$ for the mirror image of a knot $J$ with opposite orientation, $J \sharp K$ for the connected sum of two knots $J$ and $K$, and $a \cdot R$ for the principal ideal generated by $a \in R$.

\begin{thm} \label{thm_main1}
Let $J$ and $K$ be knots in $S^3$, and $\varphi \colon \Lambda \to R$ an admissible homomorphism.
If $\varphi(\Delta_J(t) \Delta_K(t)) \neq 0$ and if $\varphi(\Delta_J(t)) \cdot R + \varphi(\Delta_K(t)) \cdot R = R$, then $d(J, K) \geq n_\varphi(-J \sharp K)$.
\end{thm}

In fact, in Section \ref{sec_main} we will state and prove a slightly stronger theorem as Theorem~\ref{thm_main2} which takes positive and negative crossing changes into account.
In the case that $\Delta_J(t) = 1$ and $\varphi$ is the identity map on $\Lambda$, this theorem recovers \cite[Theorem 1.1]{BF15} by Borodzik and the first author.
In particular, they gave new obstructions for $u(J) = 2$ and $u(J) = 3$ as explained below.
It is known that lower bounds of $n(J)$ are given by the Nakanishi index~\cite{N81}, the Levine-Tristram signatures~\cite{BF14a, Le69, Mus65, Ta69, Tr69}, the Lickorish obstruction~\cite{CoLi86, Lic85}, the Murakami obstruction~\cite{Muk90} and the Jabuka obstruction~\cite{J09}.
See~\cite{BF, BF15} for the details.

\begin{cor}
Let $J$ and $K$ be knots in $S^3$.
If $\Delta_J(t) \cdot \Lambda + \Delta_K(t) \cdot \Lambda = \Lambda$, then $d(J, K) \geq n(-J \sharp K)$.
\end{cor}

\begin{cor}
Let $J$ and $K$ be knots in $S^3$.
If the greatest common divisor of $\Delta_J(t)$ and $\Delta_K(t)$ is equal to $\pm 1$, then $d(J, K) \geq n_{\Z \left[ \frac{1}{m} \right]}(-J \sharp K)$, where $m$ is the minimal positive integer in $\Delta_J(t) \cdot \Lambda + \Delta_K(t) \cdot \Lambda$.
\end{cor}

\begin{cor}
Let $J$ and $K$ be knots in $S^3$.
If the greatest common divisor of $\Delta_J(t)$ and $\Delta_K(t)$ is equal to $\pm 1$, then $d(J, K) \geq n_\R(-J \sharp K)$.
\end{cor}

Note that $n_\Z(J) \geq n_{\Z \left[ \frac{1}{m} \right]}(J) \geq n_\R(J)$ for every positive integer $m$.
Borodzik and the first author~\cite{BF14a} proved that $n_\R(J)$ is completely determined by the Levine-Tristram signatures and the nullities.

We now consider the admissible homomorphism $\varphi_{-1} \colon \Lambda \to \Z$ sending $t$ to $-1$.
Then it is well-known that $H_1^\varphi(X_J; \Z)$ is isomorphic to $H_1(\Sigma(J); \Z)$, where $\Sigma(J)$ is the double branched cover of $S^3$ along $J$.
Since $\Sigma(J)$ is a rational homology $3$-sphere, we have its linking pairing
\[ \lk_J \colon H_1(\Sigma(J); \Z) \times H_1(\Sigma(J); \Z) \to \Q / \Z. \]
It follows from \cite[Lemma 3.3]{BF15} that $\bl_{J, \varphi_{-1}}$ is isometric to $2 \lk_J$, where $2 \lk_J$ is defined by $(2 \lk_J)(v, w) = 2 \cdot \lk_J(v, w)$ for $v$, $w \in H_1(\Sigma(J); \Z)$.
Hence $n_{\varphi_{-1}}(J)$ coincides with the minimal size of an integral symmetric matrix representing $2 \lk_J$ and being congruent to the identity matrix modulo $2$.
Note also that $\varphi_{-1}(\Delta_J(t)) = \Delta_J(-1) = \pm \det(J)$.
Thus we have the following corollary.

\begin{cor} \label{cor_linking_pairing}
Let $J$ and $K$ be knots in $S^3$.
If $\det(J)$ and $\det(K)$ are coprime, then $d(J, K)$ is greater than or equal to the minimal size of an integral symmetric matrix representing $2 \lk_{-J \sharp K}$ and being congruent to the identity matrix modulo $2$.
\end{cor}

When one of the knots $J$ or $K$ is the unknot, Corollary~\ref{cor_linking_pairing} gives a lower bound of the unknotting number and, in particular, the Lickorish obstruction for $u(J) = 1$~\cite{CoLi86, Lic85}.
Borodzik and the first author~\cite[Section 6]{BF15} found examples of knots $J$ where the lower bound gives new obstructions for $u(J) = 2$ and $u(J) = 3$.
For example, they first determined $u(11_{a123}) = 3$ and $u(11_{n148}) = 3$.
See also \cite{BF} for the details.

\subsection*{Applications}

Corollary~\ref{cor_linking_pairing} provides a new computable obstruction for $d(J, K) = 2$ for knots $J$ and $K$ in $S^3$.
The idea behind our approach is the following, which is the same as in \cite[Section 5]{BF15}.
First it is well-known that the isometric type of the linking pairing $\lk_{-J \sharp K}$ for $J$ and $K$ can be calculated in terms of Seifert matrices of $J$ and $K$.
Suppose that $\lk_{-J \sharp K}$ is represented by an integral $n \times n$ matrix.
Up to congruence there exists finitely many such matrices, which furthermore in many cases can be listed explicitly.
It is then straightforward to verify whether or not $\lk_{-J \sharp K}$ can be represented by any of these matrices.

As a concrete example we show in Example~\ref{exmp_3} that the Gordian distance from the trefoil $3_1$ to the connected sum $4_1\sharp 4_1$ of two copies of the figure 8 knot equals 3, in other words, $d(3_1,4_1 \sharp 4_1)=3$. It seems like most other known lower bounds on the Gordian distance struggle to prove this result.

Next note that Corollary~\ref{cor_linking_pairing} provides the same lower bounds of $d(J, K)$, $d(rJ, K)$, $d(J, rK)$ and $d(rJ, rK)$, where we write $rJ$ for a knot $J$ with opposite orientation.
This follows from the fact that the linking pairing $\lk_{-J \sharp K}$ is invariant under the changes of the orientations of the knots, as it will be discussed in Section~\ref{subsec_obstructions}. 
Also, by definition, $d(J, K) = d(K, J) = d(mJ, mK) = d(mK, mJ)$, where we write $mJ$ for the mirror image of a knot $J$.
Therefore in describing numbers of pairs in the following computational results we do not distinguish the pairs $(J, K)$, $(rJ, K)$, $(J, rK)$ and $(rJ, rK)$, and neither the pairs $(J, K)$, $(K, J)$, $(mJ, mK)$ and $(mK, mJ)$.

Among all pairs $(J, K)$ of nontrivial prime knots with up to $10$ crossings there are $886$ pairs for which Corollary~\ref{cor_linking_pairing} shows that $d(J, K) \geq 3$, but for which any of the lower bounds of $d(J, K)$ by the signature, the Rasmussen $s$-invariant~\cite{Ra10}, the Ozsv\'ath-Szab\'o $\tau$-invariant~\cite{OS03} and the maximal rank of $H_1(\Sigma(J); \F_p)$ for all odd primes $p$ is less than $3$.
For $195$ pairs $(J, K)$ of knots among the $886$ ones, it is known that $u(J) + u(K) \leq 3$, and since $d(J, K) \leq u(J) + u(K)$, we can conclude that $d(J, K) = u(J) + u(K) = 3$.
Also, if we count the numbers of such pairs of knots, instead, for all pairs $(J, K)$ of nontrivial knots (possibly non-prime) with up to $10$ crossings, then the first one $886$ is replaced by $1696$, and the second one $195$ is by $360$.

Details are given in Section~\ref{sec_app} and on the website~\cite{FKS} by the authors.
At the time of writing we have not yet implemented the obstruction to $d(J, K) = n$ for higher values of $n$.

\subsection*{Organization}

Section~\ref{sec_pairing} provides a brief review of Blanchfield pairings of a knot and a homology $S^2 \times S^1$ and intersection pairings of a $4$-manifold with coefficients in $R$.
In Section~\ref{sec_cob} we study a certain relation among the pairings of homology $S^1 \times S^2$'s and of a cobordism between them, which is a key ingredient in the proof of the main theorem.
In Section~\ref{sec_main} we state and prove a slightly stronger theorem deducing the main theorem.
Section~\ref{sec_app} is devoted to applications, where we describe more details of our approach and computational results with examples. 

\subsection*{Convention and notation}

All knots are understood to be oriented, and all manifolds are understood to be compact, connected and oriented, unless we say specifically otherwise.
We do not assume that manifolds are smooth.
Throughout this paper $R$ is a Noetherian unique factorization domain with (possibly trivial) involution $\bar{\cdot} \colon R \to R$.
The quotient field of $R$ is denoted by $Q(R)$.

\subsection*{Acknowledgments}
The website Knotinfo: Table of knot invariants~\cite{LM}, maintained by Chuck Livingston and Allison H.~Moore, has been an invaluable tool for finding examples and testing our algorithm.
The authors would like to thank them for helpful information on data provided on the website.
The research started and was carried out while the second and third authors were visiting the University of Regensburg.
They were very grateful for the warm hospitality.

The first author was supported by the SFB 1085 ``higher invariants'' funded by the DFG.
The second author was supported by JSPS KAKENHI Grant Numbers JP18K13404 and JP18KK0380.
The third author was supported by JSPS KAKENHI Grant Numbers JP19H01785 and JP20K03596.

%%%%%%% Section 2 %%%%%%%%%%%%%%%%%%%%%%%%%%%%%%%%%%%%%%%%%%%%%%%%%%%%%%%%%%%%%
\section{Blanchfield pairings and twisted intersection pairings} \label{sec_pairing}

We begin with the definitions of the twisted Blanchfield pairings of a knot and a homology $S^1 \times S^2$ and the twisted intersection pairing of a $4$-manifold associated with an admissible homomorphism $\varphi \colon \Lambda \to R$.
We refer the reader to \cite[Chapter 2]{H12} for a thorough treatment of the classical Blanchfield pairing and to \cite[Appendix]{FKLMN20} for more details on (commutative) twisted Blanchfield pairings.

\subsection{Twisted homology and cohomology groups}

Let $X$ be a topological space equipped with an epimorphism $H_1(X; \Z) \to \Z$ admitting the corresponding infinite cyclic covering $\widetilde{X}$ of $X$.
Let $Y$ be a subspace of $X$, and we write $\widetilde{Y}$ for the preimage of $Y$ by the covering map $\widetilde{X} \to X$.
The singular chain complex $C_*(\widetilde{X}, \widetilde{Y})$ has the structure of a $\Lambda$-module, where the action by $t$ corresponds to the deck transformation on $\widetilde{X}$ by $1 \in \Z$.

Recall that a homomorphism $\varphi \colon \Lambda \to R$ is called \textit{admissible} if it satisfies the following:
\begin{enumerate}
\item $\varphi(t) \neq 1$ and $\varphi(t^{-1}) = \overline{\varphi(t)}$,
\item For every finitely generated $R$-module $L$ the module $\hom_R(L, R)$ is free.
\end{enumerate}

Let $\varphi \colon \Lambda \to R$ be an admissible homomorphism and $L$ an $R$-module, which has the structure of a $\Lambda$-module induced by $\varphi$.
For each nonnegative integer $i$ we define the \textit{$i$-th twisted homology group} $H_i^\varphi(X, Y; L)$ and the \textit{$i$-th twisted cohomology group} $H_\varphi^i(X, Y; L)$ \textit{of $(X, Y)$ associated with $\varphi$} as:
\begin{align*}
H_i^\varphi(X, Y; L) &= H_i(\overline{C_* (\widetilde{X}, \widetilde{Y})} \otimes_{\Lambda} L), \\
H_\varphi^i(X, Y; L) &= H^i(\hom_\Lambda(C_*(\widetilde{X}, \widetilde{Y}), L)).
\end{align*}
Here we denote by $\overline{C_* (\widetilde{X}, \widetilde{Y})}$ the module with the involuted $\Lambda$-structure, that is, $\overline{C_* (\widetilde{X}, \widetilde{Y})} = C_* (\widetilde{X}, \widetilde{Y})$ as abelian groups but multiplication by $a \in \Lambda$ corresponds to multiplication by $\bar{a}$.
Similarly, for an $R$-module $H$ we denote by $\overline{H}$ the module with the involuted $R$-structure.
When $Y$ is empty, we write $H_i^\varphi(X; L)$ and $H_\varphi^i(X; L)$ respectively.
Also, when $\varphi$ is the identity map on $\Lambda$ and $L = \Lambda$, then we write $H_i(X, Y; \Lambda)$ and $H^i(X, Y; \Lambda)$ respectively.

The \textit{Kronecker pairing}
\[ \kappa \colon H_\varphi^i(X, Y; L) \times H_i^\varphi(X, Y; R) \to L \]
is defined as the induced one by the sesquilinear pairing
\[ \hom_\Lambda(C_i(\widetilde{X}, \widetilde{Y}), L) \times \overline{C_i (\widetilde{X}, \widetilde{Y})} \otimes_{\Lambda} R \to L,~ (f, c \otimes a) \mapsto \bar{a}f(c) \]
for $f \in \hom_\Lambda(C_i(\widetilde{X}, \widetilde{Y}), L)$, $c \in \overline{C_i (\widetilde{X}, \widetilde{Y})}$ and $a \in R$.
Thus we have the evaluation map
\[ H_\varphi^i(X, Y; L) \to \overline{\hom_R(H_i^\varphi(X, Y; R), L)},~ f \mapsto \kappa(f, \cdot). \]

We set $R_0 = R / (\varphi(t) - 1) R$. 
If $X$ is path-connected, then the $0$-th twisted homology and cohomology groups are computed as follows (see for instance \cite[Proposition 3.1]{HS71}):
\[ H_0^\varphi(X; L) = L \otimes_R R_0,~ H_\varphi^0(X; L) = \{ v \in L ~;~ \varphi(t) v = v \}. \]
Since $\varphi(t) \neq 1$ and $R$ is an integral domain, we have
\[ \tor_1^\Lambda(H_0(X; \Lambda), R) = \tor_1^\Lambda(\Lambda / (t-1) \Lambda, R) = \{ x \in R ~;~ (\varphi(t) - 1) x = 0 \} = 0. \]
It thus follows from the universal coefficient spectral sequence~\cite[Theorem 10.90]{Ro09} that the homomorphism $H_1(X; \Lambda) \otimes_\Lambda R \to H_1^\varphi(X; R)$ is an isomorphism.

\subsection{Blanchfield pairings} \label{subsec_blanchfield}

Recall that for a finitely generated $R$-module $H$ with an exact sequence
\[ R^l \to R^m \xrightarrow{r} H \to 0, \]
where $l \geq m$, the order of $H$ is defined to be the greatest common divisor of the $m$-minors of a representation matrix of $r$, and is well-defined up to multiplication by a unit in $R$. 
(See for instance \cite{H12, Lic97}.)

Let $J$ be a knot in $S^3$.
The complement $X_J$ of an open tubular neighborhood of $J$ in $S^3$ admits an isomorphism $H_1(X_J; \Z) \to \Z$ induced by the orientation of $J$.
Recall that the Alexander polynomial $\Delta_J(t) \in \Lambda$ of $J$ is defined to be the order of the Alexander module $H_1(X_J; \Lambda)$.
Let $\varphi \colon \Lambda \to R$ be an admissible homomorphism, and suppose that $\varphi(\Delta_K(t)) \neq 0$.
Since $H_1^\varphi(X_J; R)$ is isomorphic to $H_1(X_J; \Lambda) \otimes_\Lambda R$, $H_1^\varphi(X_J; R)$ is a torsion $R$-module and its order is equal to $\varphi(\Delta_J(t))$.

We consider the following sequence of homomorphisms 
\begin{align*}
\Phi_{J, \varphi} \colon H_1^\varphi(X_J; R) &\to H_1^\varphi(X_J, \partial X_J; R) & &\to H_\varphi^2(X_J; R) \\
&\xleftarrow{\cong} H_\varphi^1(X_J; Q(R) / R) & &\to \overline{\hom_R(H_1^\varphi(X_J; R), Q(R) / R)}.
\end{align*}
Here the homomorphisms are as follows:
\begin{enumerate}
\item the first one is the inclusion induced homomorphism;
\item the second one comes from Poincar\'e duality;
\item the third one is the Bockstein homomorphism
\[ H_\varphi^1(X_J; Q(R) / R) \to H_\varphi^2(X_J; R) \]
corresponding to the short exact sequence of $\Lambda$-modules
\[ 0 \to R \to Q(R) \to Q(R) / R \to 0; \]
\item the last one is the evaluation map induced by the Kronecker pairing $\kappa$.
\end{enumerate}
All of these can be checked to be isomorphisms.
See \cite[Appendix]{FKLMN20} for the details.
Thus $\Phi_{J, \varphi}$ induces a non-singular sesquilinear pairing
\[ \bl_{J, \varphi} \colon H_1^\varphi(X_J; R) \times H_1^\varphi(X_J; R) \to Q(R) / R,~ (v, w) \mapsto \Phi_{J, \varphi}(v)(w), \]
which we call the \textit{Blanchfield pairing of $J$ associated with $\varphi$}.
It is well-known that the pairing is hermitian.
When $\varphi$ is the identity map on $\Lambda$, the pairing is the ordinary Blanchfield pairing $\bl_J$ of $J$.

We call a closed $3$-manifold $M$ equipped with an isomorphism $H_1(M; \Z) \to \Z$ a \textit{homology $S^1 \times S^2$}.
Let $M$ be a homology $S^1 \times S^2$.
We define the \textit{Alexander polynomial} $\Delta_M(t) \in \Lambda$ of $M$ to be the order of $H_1(M; \Lambda)$.
We will show in Lemma~\ref{lem_homology1} that $H_1^\varphi(M; R)$ is a torsion $R$-module.
Its order is equal to $\varphi(\Delta_M(t))$.
This can be checked for instance by using Reidemeister torsion~\cite[Proposition 3.6 and Corollary 11.9]{Tu01}.

Suppose that $\varphi(\Delta_M(t)) \neq 0$.
Similarly, the following sequence of isomorphisms 
\[ \Phi_{M, \varphi} \colon H_1^\varphi(M; R) \to H_\varphi^2(M; R) \xleftarrow{\cong} H_\varphi^1(M; Q(R) / R) \to \overline{\hom_R(H_1^\varphi(M; R), Q(R) / R)} \]
induces a non-singular hermitian sesquilinear pairing
\[ \bl_{M, \varphi} \colon H_1^\varphi(M; R) \times H_1^\varphi(M; R) \to Q(R) / R,~ (v, w) \mapsto \Phi_{M, \varphi}(v)(w), \]
which we call the \textit{Blanchfield pairing of $M$ associated with $\varphi$}.

We denote by $M_J$ the result of $0$-framed surgery of $S^3$ along $J$, which is a homology $S^1 \times S^2$.
The equipped isomorphism $H_1(M_J; \Z) \to \Z$ sends the meridional element to $1$.
It can be checked that the inclusion induced homomorphism $H_1^\varphi(X_J; R) \to H_1^\varphi(M_J; R)$ is an isometry with respect to $\bl_{J, \varphi}$ and $\bl_{M_J, \varphi}$.

\subsection{Intersection pairings}

Let $W$ be a topological $4$-manifold equipped with an epimorphism $H_1(W; \Z) \to \Z$, and $\varphi \colon \Lambda \to R$ an admissible homomorphism.
Recall that we set $R_0 = R / (\varphi(t) - 1) R$.
We denote by $G^\varphi(W; R)$ and $G(W; R_0)$ the cokernels of the inclusion induced homomorphisms $H_2^\varphi(\partial W; R) \to H_2^\varphi(W; R)$ and $H_2(\partial W; R_0) \to H_2(W; R_0)$ respectively.
It follows from the homology long exact sequence for $(W, \partial W)$ that $G^\varphi(W; R)$ and $G(W; R_0)$ are isomorphic to the kernels of the homomorphisms $H_2^\varphi(W, \partial W; R) \to H_1^\varphi(\partial W; R)$ and $H_2(W, \partial W; R_0) \to H_1(\partial W; R_0)$ respectively.

We consider the following sequence of homomorphisms
\[ \Psi_{W, \varphi} \colon H_2^\varphi(W; \Lambda) \to H_2^\varphi(W, \partial W; \Lambda) \to H_\varphi^2(W; \Lambda) \to \overline{\hom_R(H_2^\varphi(W; R), R)}. \]
Here the homomorphisms are as follows:
\begin{enumerate}
\item the first one is the inclusion induced homomorphism;
\item the second one comes from Poincar\'e duality;
\item the third one is the evaluation map induced by the Kronecker pairing $\kappa$.
\end{enumerate}
The map $\Psi_{W, \varphi}$ induces a hermitian sesquilinear pairing:
\[ H_2^\varphi(W; R) \times H_2(W; R) \to R,~ (v, w) \mapsto \widetilde{\Psi}_{W, \varphi}(v)(w), \]
which we call the \textit{intersection pairing of $W$ associated with $\varphi$}.
Since $\Psi_{W, \varphi}$ factors through $H_2^\varphi(W, \partial W; \Lambda)$, the intersection pairing induces a pairing
\[ G^\varphi(W; R) \times G^\varphi(W; R) \to R. \]

Similarly, considering the composition of $\varphi$ and the quotient map $R \to R_0$, we have the intersection pairings of $W$ over $R_0$:
\[ H_2(W; R_0) \times H_2(W; R_0) \to R_0,~ G(W; R_0) \times G(W; R_0) \to R_0. \]
When $\varphi$ is the identity map on $\Lambda$, $R_0 = \Z$ and the first pairing is the ordinary intersection pairing of $W$. 

%%%%%%% Section 3 %%%%%%%%%%%%%%%%%%%%%%%%%%%%%%%%%%%%%%%%%%%%%%%%%%%%%%%%%%%%%
\section{Cobordisms between homology $S^1 \times S^2$'s} \label{sec_cob}

We describe a relation of twisted Blanchfield pairings of homology $S^1 \times S^2$'s and twisted intersection pairings of a certain cobordism between them.

\subsection{Tame cobordisms}

Let $M$ and $N$ be homology $S^1 \times S^2$'s, and $\varphi \colon \Lambda \to R$ an admissible homomorphism with $\varphi(\Delta_M(t) \Delta_N(t)) \neq 0$. 
We call a topological $4$-manifold $W$ equipped with an isomorphism $H_1(W; \Z) \to \Z$ a \textit{$\varphi$-tame cobordism from $M$ to $N$} if it satisfies the following:
\begin{enumerate}
\item $\partial W$ is a disjoint union of $-M$ and $N$.
\item The inclusion induced homomorphisms $H_1(M; \Z) \to H_1(W; \Z)$ and $H_1(N; \Z) \to H_1(W; \Z)$ are isomorphisms, and compatible with the equipped isomorphisms, that is, the following diagram commutes:
\[ \xymatrix{
H_1(M; \Z) \ar[r] \ar[rd] & H_1(W; \Z) \ar[d] & H_1(N; \Z) \ar[l] \ar [ld] \\
& \Z &
} \]
\item $H_1^\varphi(W; R) = 0$.
\end{enumerate}
Note that $H_1^\varphi(\partial W; R) = H_1^\varphi(M; R) \oplus H_1^\varphi(N; R)$, and it is equipped with the Blanchfield pairing associated with $\varphi$
\[ \bl_{-M, \varphi} \oplus \bl_{N, \varphi} \colon H_1^\varphi(\partial W; R) \times H_1^\varphi(\partial W; R) \to Q(R) / R. \]

The following proposition will be the key ingredient in the proof of the main theorem in Section~\ref{sec_main}, and proved in the following subsection.
For a matrix $A$ over $R$ we write $A_0$ for its image under the quotient map $R \to R_0$.
\begin{prop} \label{prop_comparing}
Let $M$ and $N$ be homology $S^1 \times S^2$'s, $\varphi \colon \Lambda \to R$ an admissible homomorphism with $\varphi(\Delta_M(t) \Delta_N(t)) \neq 0$, and $W$ a $\varphi$-tame cobordism from $M$ to $N$.
Then $G^\varphi(W; R)$ is a free $R$-module.
Furthermore, every hermitian matrix $A$ over $R$ representing the twisted intersection pairing associated with $\varphi$ on $G^\varphi(W; R)$ satisfies the following:
\begin{enumerate}
\item $\lambda(A)$ is isometric to $\bl_{-M, \varphi} \oplus \bl_{N, \varphi}$.
\item $A_0$ represents the intersection paring on $G(W; R_0)$.
\end{enumerate}
\end{prop}

\subsection{Proof of Proposition~\ref{prop_comparing}}

Let $M$ be a homology $S^1 \times S^2$.
The following lemma is standard.

\begin{lem} \label{lem_homology1}
The following hold:
\begin{enumerate}
\item $H_1^\varphi(M; R)$ is a torsion $R$-module, and multiplication by $(\varphi(t) - 1)$ defines an automorphism on $H_1^\varphi(M; R)$.
\item $H_2^\varphi(M; R)$ is isomorphic to $R_0$, and the homomorphism $H_2^\varphi(M; R) \otimes_R R_0 \to H_2(M; R_0)$ is an isomorphism.
\end{enumerate}
\end{lem}

\begin{proof}
We first show (1).
We consider the universal coefficient spectral sequence $E_{p, q}^2 = \tor_p^R(H_q^\varphi(M; R), R_0)$ with the differential map $d^2 \colon E_{p, q}^2 \to E_{p - 2, q + 1}^2$ converging to $H_{p + q}(M; R_0)$ \cite[Theorem 10.90]{Ro09}.
Since $R_0$ admits a free resolution of length $1$, we have $E_{2, q}^2 = 0$ for all $q$.
In particular, the differential map $d^2 \colon E_{2, 0}^2 \to E_{0, 1}^2$ is the $0$-map.
Hence $H_1(M; R_0)$, which is isomorphic to $R_0$, contains $H_1^\varphi(M; R) \otimes_R R_0$ and $E_2^{1, 0}$ as direct summands.
Here since $H_0^\varphi(M; R)$ is isomorphic to $R_0$, so is $E_2^{1, 0}$.
It thus follows that $H_1^\varphi(M; R) \otimes_R R_0 = 0$, which shows that multiplication by $(\varphi(t) - 1)$ defines an automorphism on $H_1^\varphi(M; R)$.
Let $x_1$, $\dots$, $x_m$ be generators of $H_1^\varphi(M; R)$.
Then we may write $x_i = (\varphi(t)-1) \sum_{j=1}^m \lambda_{ij} x_j$ for some $\lambda_{ij} \in R$.
Rearranging this we may write $\sum_{j=1}^m \mu_{ij} x_j = 0$ for $1 \leq i \leq m$, where $\mu_{ij} = (\varphi(t)-1) \lambda_{ij} - \delta_{ij}$.
We set $\Delta = \det (\mu_{ij}) \in R$.
Then $\Delta x_i = 0$ for $1 \leq i \leq m$, and $\Delta H_1^\varphi(M; \R) = 0$.
Since $\Delta$ is equal to $\det (-\delta_{ij}) = (-1)^m$ modulo $(\varphi(t) - 1)$, $\Delta \neq 0$, which shows that $H_1^\varphi(M; R)$ is a torsion $R$-module.

We next show (2).
Poincar\'e duality implies that $H_2^\varphi(M; R)$ is isomorphic to $H_\varphi^1(M; R)$.
We consider the universal coefficient spectral sequence $E_2^{p, q} = \overline{\ext_R^q(H_p^\varphi(M; R), R)}$ with the differential map $d_2 \colon E_2^{p, q} \to E_2^{p - 1, q + 2}$ converging to $H_\varphi^{p + q}(M; R)$ \cite[Theorem 2.3]{Le77}.
Note that here we need the involuted $R$-structure of $\ext_R^q(H_p^\varphi(M; R), R)$ since we apply the universal coefficient spectral sequence to the chain complex $\overline{C_*(\widetilde{M})} \otimes_\Lambda R$ and then the cochain complex $\hom_\Lambda(C_*(\widetilde{M}), R)$ is identified with $\hom_R(\overline{C_*(\widetilde{M})} \otimes_\Lambda R, R)$ with the involuted $R$-structure by the Kronecker pairing.
Since $H_0^\varphi(M; R)$ is isomorphic to $R_0$, so is $E_2^{0, 1}$.
By (1) we have $E_2^{1, 0} = 0$.
Thus $H_2^\varphi(M; R)$ is isomorphic to $R_0$.
Now we again consider the universal coefficient spectral sequence $E_{p, q}^2 = \tor_p^R(H_q^\varphi(M; R), R_0)$ converging to $H_{p + q}(M; R_0)$.
Recall that we have $E_{2, q}^2 = 0$ for all $q$.
In particular, the differential map $d^2 \colon E_{2, 1}^2 \to E_{0, 2}^2$ is the $0$-map.
Also, since $H_1^\varphi(M; \R)$ is $(\varphi(t)-1)$-torsion free by (1), $E_{1, 1}^2 = 0$.
Thus the homomorphism $H_2^\varphi(M; R) \otimes_R R_0 \to H_2(M; R_0)$ is an isomorphism.
\end{proof}

Let $M$ and $N$ be homology $S^1 \times S^2$'s, $\varphi \colon \Lambda \to R$ an admissible homomorphism with $\varphi(\Delta_M(t) \Delta_N(t)) \neq 0$, and $W$ a $\varphi$-tame cobordism $W$ from $M$ to $N$.

\begin{lem} \label{lem_dual}
The following hold:
\begin{enumerate}
\item The evaluation map $H_\varphi^2(W; R) \to \overline{\hom_R(H_2^\varphi(W; R), R)}$ induced by the Kronecker pairing $\kappa$ is an isomorphism.
\item The homomorphism $\hom_\Lambda(G^\varphi(W; R), R) \to \hom_R(H_2^\varphi(W; R), R)$ induced by the quotient map $H_2^\varphi(W; R) \to G^\varphi(W; R)$ is an isomorphism.
\end{enumerate}
\end{lem}

\begin{proof}
We first show (1).
As in the proof of Lemma~\ref{lem_homology1} we consider the universal coefficient spectral sequence $E_2^{p, q} = \overline{\ext_R^q(H_p^\varphi(W; R), R)}$ with the differential map $d_2 \colon E_2^{p, q} \to E_2^{p - 1, q + 2}$ converging to $H_\varphi^{p + q}(W; R)$ \cite[Theorem 2.3]{Le77}.
Since $H_0^\varphi(W; R)$ is isomorphic to $R_0$, which admits a free resolution of length $1$, we have $E_2^{0, 2} = 0$.
Since $H_1(W; \Lambda) = 0$, we see at once that $E_2^{1, q} = 0$ for all $q$.
In particular, we have $E_2^{1, 1} = 0$, and see that the differential map $d_2 \colon E_2^{2, 0} \to E_2^{1, 2}$ is the $0$-map.
Thus (1) follows.

Since $H_2^\varphi(\partial W; R) = H_2^\varphi(M; R) \oplus H_2^\varphi(N; R)$ is a torsion module by Lemma~\ref{lem_homology1} (2), every homomorphism $H_2^\varphi(W; R) \to R$ induces a homomorphism $G^\varphi(W; R) \to R$, and (2) follows.
\end{proof}

We now have the following sequence of the isomorphisms:
\[ D_{W, \varphi} \colon H_2^\varphi(W, \partial W; R) \to H_\varphi^2(W; R) \to \overline{\hom_R(H_2^\varphi(W; R), R)} \xleftarrow{\cong} \overline{\hom_R(G^\varphi(W; R), R)}, \]
Here the isomorphisms are as follows:
\begin{enumerate}
\item the first one comes from Poincar\'e duality;
\item the second one is the isomorphism as in Lemma~\ref{lem_dual} (1);
\item third ones is the isomorphism as in Lemma~\ref{lem_dual} (2).
\end{enumerate}

\begin{lem} \label{lem_homology2}
The following hold:
\begin{enumerate}
\item $H_2^\varphi(W, \partial W; R)$ is a free $R$-module.
\item $G^\varphi(W; R)$ is a free $R$-module.
\item The homomorphism $G^\varphi(W; R) \otimes_R R_0 \to G(W; R_0)$ is an isomorphism.
\end{enumerate}
\end{lem}

\begin{proof}
We have the isomorphism $D_{W, \varphi} \colon H_2^\varphi(W, \partial W; R) \to \overline{\hom_R(G^\varphi(W; R), R)}$.
It follows from the second property of the admissible homomorphism $\varphi$ that these $R$-modules are free, which shows (1).

We next show (2).
Since $H_1^{\varphi}(W; R) = 0$, the homology long exact sequence for $(W, \partial W)$ descends to the following short exact sequence:
\[ 0 \to G^\varphi(W; R) \to H_2^\varphi(W, \partial W; R) \to H_1^\varphi(\partial W; R) \to 0.  \] 
By (1) we can find an isomorphism $f \colon R^r \to H_2^\varphi(W, \partial W; R)$, where $r = \rank_R H_2^\varphi(W, \partial W; R)$.
Let $g \colon R^s \to G^\varphi(W; R)$ be an epimorphism such that $s$ is minimal among such epimorphisms.
Thus we have a presentation matrix $A$ of $H_1^\varphi(\partial W; R)$ so that the following diagram commutes:
\[ \xymatrix{
& R^s \ar[r]^{A \cdot} \ar[d]_g & R^r \ar[r] \ar[d]_f & H_1^\varphi(\partial W; R) \ar[r] \ar@{=}[d] & 0 \\
0 \ar[r] & G^\varphi(W; R) \ar[r] & H_2^\varphi(W, \partial W; R) \ar[r] & H_1^\varphi(\partial W; R) \ar[r] & 0
} \]

It is well-known that $H_1(M_J; \Lambda)$, isomorphic to $H_1(X_J; \Lambda)$, for a knot $J$ in $S^3$ admits a nonsingular square presentation matrix, and so does $H_1^\varphi(M_J; R)$, isomorphic to $H_1(M_J; \Lambda) \otimes_\Lambda R$.
Since $M$ and $N$ can be seen as the results of $0$-framed surgery of a homology $3$-sphere $\Sigma$ along some knots in $\Sigma$, similarly, $H_1(M; \Lambda)$ and $H_1(N; \Lambda)$ admits a nonsingular square presentation matrix, and so do $H_1^\varphi(M; R)$ and $H_1^\varphi(N; R)$, isomorphic to $H_1(M; \Lambda) \otimes_\Lambda R$ and $H_1(N; \Lambda) \otimes_\Lambda R$ respectively.
Let $B$ and $C$ be such presentation matrices of $H_1^\varphi(M; R)$ and $H_1^\varphi(N; R)$.
Then $B \oplus C$ is another presentation matrix of $H_1^\varphi(\partial W; R) = H_1^\varphi(M; R) \oplus H_1^\varphi(N; R)$.
Two presentation matrices of a same module can be transformed into each other by the following $4$ operations and their inverses~\cite[Theorem 6.1]{Lic97}:
\begin{enumerate}
\item Permutation of rows and columns;
\item Replacement of the matrix $P$ by $P \oplus (1)$;
\item Addition of an extra column of zeros to the matrix;
\item Addition of a scalar multiple of a row (or column) to another row (or column).
\end{enumerate}
It thus follows that the difference between the minimal number of generators of the module generated by columns of $A$ and the number of rows of $A$ is equal to that for $B \oplus C$, which is $0$.
Hence we have $s = r$, and so $A$ is a square matrix.
Since $\det A = \det B \det C \neq 0$, we have the following commutative diagram of short exact sequences:
\[ \xymatrix{
0 \ar[r] & R^r \ar[r]^{A \cdot} \ar[d]_g & R^r \ar[r] \ar[d]_f & H_1^\varphi(\partial W; R) \ar[r] \ar@{=}[d] & 0 \\
0 \ar[r] & G^\varphi(W; R) \ar[r] & H_2^\varphi(W, \partial W; R) \ar[r] & H_1^\varphi(\partial W; R) \ar[r] & 0
} \]
It now follows from the five lemma that $g$ is an isomorphism, and (2) follows.

Finally, we show (3).
It suffices to see that the homomorphisms $H_2^\varphi(\partial W; R) \otimes_R R_0 \to H_2(\partial W; R_0)$ and $H_2^\varphi(W; R) \otimes_R R_0 \to H_2(W; R_0)$ are isomorphisms.
By Lemma~\ref{lem_homology1} (2) the first one is an isomorphism.
We consider the universal coefficient spectral sequence $E_{p, q}^2 = \tor_p^R(H_q^\varphi(W; R), R_0)$ with the differential map $d^2 \colon E_{p, q}^2 \to E_{p - 2, q + 1}^2$ converging to $H_{p + q}(W; R_0)$ \cite[Theorem 10.90]{Ro09}.
Since $R_0$ admits a free resolution of length $1$, we have $E_{2, 0}^2 = E_{2, 1}^2 = 0$.
In particular, the differential map $d_2 \colon E_{2, 1}^2 \to E_{0, 2}^2$ is the $0$-map.
Also, since $H_1^\varphi(W; R) = H_1^\varphi(M; R) \oplus H_1^\varphi(N; R)$ is $(\varphi(t) -1)$-torsion free by Lemma~\ref{lem_homology1} (1), $E_{1, 1}^2 = 0$.
Thus the homomorphism $H_2^\varphi(W; R) \otimes_R R_0 \to H_2(W; R_0)$ is an isomorphism, and (3) follows.
\end{proof}

\begin{lem} \label{lem_h2}
The homomorphism $H_\varphi^2(W, \partial W; Q(R)) \to H_\varphi^2(W; Q(R))$ is an isomorphism.
\end{lem}

\begin{proof}
Since $H_1^\varphi(W; Q(R)) = H_2^\varphi(\partial W; Q(R)) = 0$ by Lemma~\ref{lem_homology1} (1), it follows from the homology long exact sequence for $(W, \partial W)$ that the homomorphism $H_2^\varphi(W; Q(R)) \to H_2^\varphi(W, \partial W; Q(R))$ is an isomorphism.
By Poincar\'e duality we have the desired isomorphism.
\end{proof}

We consider the following sequence of homomorphisms
\begin{align*}
\overline{\Psi}_{W, \varphi} \colon H_2^\varphi(W, \partial W; R) &\to H_2^\varphi(W, \partial W; Q(R)) & &\to H_\varphi^2(W; Q(R)) \\
&\xleftarrow{\cong} H_\varphi^2(W, \partial W; Q(R)) & &\to \overline{\hom_R(H_2^\varphi(W, \partial W; R), Q(R))}.
\end{align*}
Here the homomorphisms are as follows:
\begin{enumerate}
\item the first one is a localization map;
\item the second one comes from Poincar\'e duality;
\item the third one is the isomorphism as in Lemma~\ref{lem_h2};
\item the last one is the evaluation map induced by the Kronecker pairing $\kappa$.
\end{enumerate}
As the twisted intersection pairing of $W$ the map $\overline{\Psi}_{W, \varphi}$ induces a hermitian sesquilinear pairing
\[ H_2^\varphi(W, \partial W; R) \times H_2^\varphi(W, \partial W; R) \to Q(R),~ (v, w) \mapsto \overline{\Psi}_{W, \varphi}(v)(w), \]
which we call the \textit{twisted intersection pairing} of $(W, \partial W)$.

The following lemma is proved totally parallel to one of the claims in the proof of \cite[Lemma 2.8]{BF15}.
\begin{lem} \label{lem_commuting}
The twisted intersection pairings of $W$ and $(W, \partial W)$ and the twisted Blanchfield pairings of $M$ and $N$ associated with $\varphi$ fit into the following commutative diagram:
\[ \xymatrix{
H_2^\varphi(W; R) \times H_2^\varphi(W; R) \ar[rr] \ar[d] & & R \ar[d] \\
H_2^\varphi(W, \partial W; R) \times H_2^\varphi(W, \partial W; R) \ar[rr] \ar[d] & & Q(R) \ar[d] \\
H_1^\varphi(\partial W; R) \times H_1^\varphi(\partial W; R) \ar[rr]^-{\bl_{-M, \varphi} \oplus \bl_{N, \varphi}} & & Q(R) / R.
} \]
\end{lem}

We are now in a position to prove Proposition~\ref{prop_comparing}.

\begin{proof}[Proof of Proposition~\ref{prop_comparing}]
First, $H_2^\varphi(W, \partial W; R)$ and $G^\varphi(W; R)$ are free $R$-modules by Lemma~\ref{lem_homology2} (1) and (2).
We pick an arbitrary basis $\mathcal{B}$ of $G^\varphi(W; R)$, and set $A$ to be the hermitian matrix representing the twisted intersection pairing associated with $\varphi$ on $G^\varphi(W; R)$ with respect to $\mathcal{B}$.
By Lemma~\ref{lem_homology2} (3) we see that $A$ satisfies (2).

We write $\mathcal{C}$ for the basis of $H_2^\varphi(W, \partial W; R)$ dual to $\mathcal{B}$ by the isomorphism $D_{W, \varphi} \colon H_2^\varphi(W, \partial W; R) \to \overline{\hom_R(G^\varphi(W; R), R)}$.
Then the homomorphism $G^\varphi(W; R) \to H_2^\varphi(W, \partial W; R)$ is given by the map $v \mapsto A v$ with respect to $\mathcal{B}$ and $\mathcal{C}$.
By Lemma~\ref{lem_commuting} we see that $A^{-1}$ represents the twisted intersection pairing of $(W, \partial W)$ associated with $\varphi$ with respect to $\mathcal{C}$, and we have the following commutative diagram of the pairings: 
\[ \xymatrix{
R^s \times R^s \ar[rr]^{(v, w) \mapsto \bar{v}^T A w} \ar[d]^{(v, w) \mapsto (A v, A w)} & & R \ar[d] \\
R^s \times R^s \ar[rr]^{(v, w) \mapsto \bar{v}^T A^{-1} w} \ar[d] & & Q(R) \ar[d] \\
H_1^\varphi(\partial W; R) \times H_1^\varphi(\partial W; R) \ar[rr]^-{\bl_{-M, \varphi} \oplus \bl_{N, \varphi}} & & Q(R) / R,
} \]
where $s = \rank_R G^\varphi(W; R) = \rank_{R_0} G(W; R_0)$.
Since the vertical maps form short exact sequences, we now see that $A$ satisfies (1).
\end{proof}

%%%%%%% Section 4 %%%%%%%%%%%%%%%%%%%%%%%%%%%%%%%%%%%%%%%%%%%%%%%%%%%%%%%%%%%%%
\section{Proof of the main theorem} \label{sec_main}

We now prove Theorem~\ref{thm_main1}.
In fact, we state and prove a slightly stronger theorem which takes positive and negative crossing changes into account.
A crossing change for a knot in $S^3$ is called positive (resp. negative) if it turns a negative (resp. positive) crossing in a knot diagram into positive (resp. negative) one.

Theorem~\ref{thm_main1} is a direct corollary of the following theorem.
Recall that for a matrix $A$ over $R$ we write $A_0$ for its image under the quotient map $R \to R_0$.  
\begin{thm} \label{thm_main2}
Let $J$ and $K$ be knots in $S^3$ such that $J$ can be turned into $K$ by $n_+$ positive crossing changes and $n_-$ negative ones, and let $\varphi \colon \Z[t^{\pm}] \to R$ be an admissible homomorphism.
If $\varphi(\Delta_J(t) \Delta_K(t)) \neq 0$ and if $\varphi(\Delta_J(t)) R$ and $\varphi(\Delta_K(t)) R$ are coprime, then there exists a hermitian matrix $A$ of size $n_+ + n_-$ over $R$ satisfying the following:
\begin{enumerate}
\item $\lambda(A)$ is isometric to $\bl_{-J \sharp K, \varphi}$.
\item $A_0$ is the diagonal matrix with $n_+$ entries of $1$ and $n_-$ ones of $-1$.
\end{enumerate}
\end{thm}

Considering all knots $J$ with $\Delta_J(t) = 1$ and the identity map on $\Lambda$ as $\varphi$ in Theorem~\ref{thm_main2}, we recover \cite[Theorem 1.1]{BF15}. 

The following proposition, combined with Proposition~\ref{prop_comparing}, implies Theorem~\ref{thm_main2}.
Note that we have
\[ \bl_{-M_{J, \varphi}} \oplus \bl_{M_{K, \varphi}} = \bl_{-J, \varphi} \oplus \bl_{K, \varphi} = \bl_{-J \sharp K, \varphi}. \]

\begin{prop} \label{prop_construction}
Let $J$ and $K$ be knots in $S^3$ such that $J$ can be turned into $K$ by $n_+$ positive crossing changes and $n_-$ negative ones, and let $\varphi \colon \Lambda \to R$ be an admissible homomorphism.
If $\varphi(\Delta_J(t) \Delta_K(t)) \neq 0$ and if $\varphi(\Delta_J(t)) \cdot R + \varphi(\Delta_K(t)) \cdot R = R$, then there exists a smooth $\varphi$-tame cobordism $W$ from $M_J$ to $M_K$ satisfying the following:
\begin{enumerate}
\item $G(W; R_0)$ is a free $R_0$-module of rank $n_+ + n_-$.
\item The intersection paring on $G(W; R_0)$ is represented by the diagonal matrix with $n_+$ entries of $1$ and $n_-$ ones of $-1$.
\end{enumerate} 
\end{prop}

\begin{proof}
We first construct a smooth cobordism $W$ from $M_J$ to $M_K$.
We write $s = n_+ + n_-$.
We recall that a positive (resp. negative) crossing change of a knot is realized by performing $1$-surgery (resp. $(-1)$-surgery) along the boundary of an embedded disk in $S^3$ intersecting the knot in precisely two points with opposite orientations.
Such embedded disks in $S^3$ corresponding to crossing changes can be picked to be disjoint.
Thus there exist simple closed curves $c_1, \dots, c_{n_+}$ with $1$-framings and $c_{n_+ + 1}, \dots, c_s$ with $(-1)$-framings in $X_J$ satisfying the following:
\begin{enumerate}
\item $c_1, \dots, c_s$ form an unlink in $S^3$.
\item The linking number $\lk(c_i, J)$ of $c_i$ and $J$ is $0$ for each $i$.
\item The manifold obtained from $X_J$ by Dehn surgery along $c_1, \dots, c_s$ is diffeomorphic to $X_K$.   
\end{enumerate}
Viewing $c_1, \dots, c_s$ also as lying in $M_J$, we define $W$ to be the result of adding $s$ $2$-handles to $M_J \times [0, 1]$ along $c_1 \times \{ 1 \}, \dots, c_s \times \{ 1 \}$.
It follows from the construction of $W$ that $\partial W$ is a disjoint union of $-M_J$ and $M_K$.
Note that we can think of the cobordism $-W$ from $-M_K$ to $M_J$ also as the result of adding $s$ $2$-handles to $M_K \times [0, 1]$ along such simple closed curves in $M_K \times \{ 1 \}$.

We next show that $W$ is a $\varphi$-tame cobordism from $M_J$ to $M_K$.
Since the curves $c_1, \dots, c_s$ are nullhomologous in $M_J$, the inclusion induced homomorphism $H_1(M_J; \Z) \to H_1(W; \Z)$ is an isomorphism, and, similarly, so is the other one $H_1(M_K; \Z) \to H_1(W; \Z)$.
Since the images in $H_1(W; \Z)$ of the meridional elements of $H_1(M_J; \Z)$ and $H_1(M_K; \Z)$ coincide, $W$ is equipped with an isomorphism $H_1(W; \Z) \to \Z$ such that the following diagram commutes:
\[ \xymatrix{
H_1(M_J; \Z) \ar[r] \ar[rd] & H_1(W; \Z) \ar[d] & H_1(M_K; \Z) \ar[l] \ar [ld] \\
& \Z &
} \]
It follows from the Mayer-Vietoris homology exact sequence that the inclusion induced homomorphism $H_1^\varphi(M_J; R) \to H_1^\varphi(W; R)$ is surjective, and, similarly, so is the other one $H_1^\varphi(M_K; R) \to H_1^\varphi(W; R)$.
Since $\varphi(\Delta_J(t))$ and $\varphi(\Delta_K(t))$ annihilate $H_1^\varphi(M_J; R)$ and $H_1^\varphi(M_K; R)$ respectively, they also do $H_1^\varphi(W; R)$.
The assumption that $\varphi(\Delta_J(t)) R$ and $\varphi(\Delta_K(t)) R$ are coprime implies that there exist $a$, $b \in R$ such that $1 = a \varphi(\Delta_J(t)) + b \varphi(\Delta_K(t))$, which thus annihilates $H_1^\varphi(W; R)$.
Hence we have $H_1^\varphi(W; R) = 0$.

Finally, we show that $W$ satisfies (1) and (2) as in the statement.
Recall that $c_1, \dots, c_s$ form an unlink in $S^3$, and $\lk(c_i, J) = 0$ for each $i$.
In particular, $c_1, \dots, c_s$ are nullhomologous in $M_J$.
Now we can find disjoint subsurfaces $F_1', \dots, F_s'$ in $M_J \times [0, 1]$ with $\partial F_i' = c_i \times \{ 1 \}$ for each $i$.
For example, such a subsurface $F_i'$ can be constructed in $M_J \times \left( 1 - \frac{i}{s}, 1 \right]$ for each $i$ as the union of $c_i \times \left[ 1-\frac{i-1}{s}, 1 \right]$,a punctured disk in $M_J \times \left\{ 1 - \frac{i-1}{s} \right\}$ whose boundary is $c_i \times \left\{ 1 - \frac{i-1}{s} \right\}$, and tubes in $M_J \times \left( 1 - \frac{i}{s}, 1 - \frac{i-1}{s} \right]$ along arcs in $J$.
We write $F_i$ for the union of $F_i'$ and the core of the $2$-handle attached to $c_i$ for each $i$.
It follows from the Mayer-Vietoris homology exact sequence that (1) is satisfied, and the images of the integral homology classes of the closed subsurfaces $F_1, \dots, F_s$ under the composition of maps $H_2(W; \Z) \to H_2(W; R_0) \to G(W; R_0)$ form a basis of $G(W; R_0)$.
It is straightforward to see that
\[ F_i \cdot F_j = 0 ~\text{for $i \neq j$},~
F_i \cdot F_i =
\begin{cases}
1 & \text{for $i = 1, \dots, n_+$}, \\
-1 & \text{for $i = n_+ + 1, \dots, s$},
\end{cases}
\]
and (2) follows from the following commutative diagram of the intersection pairings of $W$:
\[ \xymatrix{
H_2(W; \Z) \times H_2(W; \Z) \ar[r] \ar[d] & \Z \ar[d] \\
H_2(W; R_0) \times H_2(W; R_0) \ar[r] & R_0.
} \]
\end{proof}

%%%%%%% Section 5 %%%%%%%%%%%%%%%%%%%%%%%%%%%%%%%%%%%%%%%%%%%%%%%%%%%%%%%%%%%%%
\section{Applications and examples} \label{sec_app}

Applying Theorem~\ref{thm_main2} and Corollary~\ref{cor_linking_pairing}, we formulate obstructions for $d(J, K) \leq 2$ for knots $J$ and $K$ in $S^3$, and describe our computational results obtained for all pairs of nontrivial knots (possibly non-prime) with up to $10$ crossings.
The full details of the results are available on the website \cite{FKS}.

\subsection{Obstructions for $d(J, K) \leq 2$} \label{subsec_obstructions}

First as an application of Theorem~\ref{thm_main2} we present obstructions for $d(J, K) = 1$ for knots $J$ and $K$ in $S^3$.

The following theorem, which generalizes the unknotting number one obstruction by Murakami~\cite{Muk90}, Fogel~\cite{F93} and John Rickard, is easily deduced from Theorem~\ref{thm_main2}.

\begin{prop} \label{prop_obstruction1}
Let $J$ and $K$ be knots in $S^3$ such that $J$ can be turned into $K$ by a single crossing change, where we set $\epsilon = 1$ if the crossing change is positive and $\epsilon = -1$ otherwise.
We pick representatives $\Delta_J(t)$, $\Delta_K(t) \in \Lambda$ of the Alexander polynomials so that $\Delta_J(1) = \Delta_K(1) = 1$.
Let $\varphi \colon \Lambda \to R$ be an admissible homomorphism.
If $\varphi(\Delta_J(t) \Delta_K(t)) \neq 0$ and if $\varphi(\Delta_J(t)) \cdot R + \varphi(\Delta_K(t)) \cdot R = R$, then there exists a generator $g$ of $H_1^\varphi(X_{-J \sharp K}; R)$ such that
\[ \bl_{-J \sharp K, \varphi}(g, g) = \frac{\epsilon}{\varphi(\Delta_J(t) \Delta_K(t))} \in Q(R) / R. \]
\end{prop}

By considering the admissible homomorphism $\varphi_{-1} \colon \Lambda \to \Z$ sending $t$ to $-1$, we have the following corollary, which generalizes the unknotting number one obstruction by Lickorish~\cite{Lic85}.

\begin{cor} \label{cor_obstruction1}
Let $J$ and $K$ be knots in $S^3$ such that $J$ can be turned into $K$ by a single crossing change, where we set $\epsilon = 1$ if the crossing change is positive and $\epsilon = -1$ otherwise.
If $\det(J)$ and $\det(K)$ are coprime, then there exists a generator $g$ of $H_1(\Sigma(-J \sharp K); \Z)$ such that
\[ \lk_{-J \sharp K, \varphi}(g, g) = \frac{2 \epsilon}{\det(J) \det(K)} \in \Z /  \det(J) \det(K) \Z. \]
\end{cor}

\begin{proof}
It follows from \cite[Lemma 3.3]{BF15} that $\bl_{J, \varphi_{-1}}$ is isometric to $2 \lk_J$.
Note also that $\varphi_{-1}(\Delta_J(t)) = \Delta_J(-1) = \pm \det(J)$.
It thus follows from Proposition~\ref{prop_obstruction1} that there exists a generator $h$ of $H_1(\Sigma(-J \sharp K); \Z)$ such that
\[ 2 \lk_{-J \sharp K, \varphi}(h, h) = \frac{\epsilon}{\det(J) \det(K)} \in \Z /  \det(J) \det(K) \Z. \]
We set $g = 2 h$.
Since $\det(J) \det(K)$ is an odd number, $g$ is also a generator of $H_1(\Sigma(-J \sharp K); \Z)$, and it has the required properties. 
\end{proof}

Next as an application of Corollary~\ref{cor_linking_pairing} we formulate an obstruction for $d(J, K) = 2$ for knots $J$ and $K$ in $S^3$.

Before stating the obstruction we recall a classification of symmetric integral $2 \times 2$-matrices up to congruence.
The full classification has been already known to Gau\ss.
The following is a slightly weaker result.
We refer to Conway and Sloane~\cite[Section 15.3]{CS99} for an excellent exposition and also to \cite[Lemma 5.2]{BF15} for a proof.

\begin{lem} \label{lem_2by2-matrices}
Let $C$ be a nonsingular symmetric integral $2 \times 2$-matrix.
Then, either $C$ is congruent to a matrix of the form
\[
\begin{pmatrix}
a & c \\
c & b
\end{pmatrix}
\]
such that
\begin{enumerate}
\item $0 < |a| \leq |b| \leq |\det(C)|$,
\item $c \in \{ 0, \dots, \lfloor \frac{|a|}{2} \rfloor \}$,
\end{enumerate}
or $C$ is congruent to a matrix of the form
\[
\begin{pmatrix}
a & c \\
c & 0
\end{pmatrix}
\]
with $c^2 = \det(C)$, $c \geq 0$ and $a \in \{ -c, \dots, c \}$.
\end{lem}

For a nonzero integer $d$ we denote by $\mathcal{C}_d$ the set of matrices which are of such forms as in Lemma~\ref{lem_2by2-matrices} and which are congruent to the identity matrix modulo $2$.
Now Corollary~\ref{cor_linking_pairing} and Lemma~\ref{lem_2by2-matrices} immediately imply the following proposition.
See \cite[Proposition 5.3]{BF15} for a corresponding proposition in the case of $\Delta_J(t) = 1$.

\begin{prop} \label{prop_obstruction2}
Let $J$ and $K$ be knots in $S^3$ with $d(J, K) \leq 2$.
If $\det(J)$ and $\det(K)$ are coprime, then there exists $C \in \mathcal{C}_d \cup \mathcal{C}_{-d}$, where $d = \det(J) \det(K)$, such that $\lambda(C)$ is isometric to $2 \lk_{-J \sharp K}$.
\end{prop}

The following lemma gives an elementary way to check whether $2 \lk_J$ is isometric to $\lambda(C)$ for a given integral $2 \times 2$-matrix $C$.

\begin{lem}[{\cite[Lemma 5.4]{BF15}}]
Let $J$ and $K$ be knots in $S^3$ and $C$ a symmetric integral $2 \times 2$-matrix with $\det(C) = \pm \det(J) \det(K)$.
Then $\lambda(C)$ is isometric to $2 \lk_{-J \sharp K}$ if and only if there exist generators $v_1$, $v_2$ of $H_1(\Sigma(-J \sharp K))$ such that
\[ 2 \lk_{-J \sharp K}(v_1, v_2) = \text{$(i, j)$-entry of $C^{-1}$}~ \in \Q / \Z \]
for any $i$, $j \in \{ 1, 2 \}$.
\end{lem}

Similarly, we can extend our approach using the linking pairing to give an obstruction for $d(J, K) = m$ for an arbitrary positive integer $m$.
See \cite[Section 5.3]{BF15} for the obstruction for $\lk_J$ to be isometric to $\lambda(C)$ for a positive-definite symmetric $m \times m$-matrix $C$.

Note that our approach using the linking pairing does not see orientations of knots.
First, it is well-known that the linking pairing is additive under the connected sum of knots:
$\lk_{J \sharp K} = \lk_J \oplus \lk_K$.
Second, the linking pairing is invariant under the change of the orientation of a knot:
$\lk_{rJ} = \lk_J$, where we write $rJ$ for a knot $J$ with opposite orientation.
This follows from the facts that $\lk_J$ is an invariant of the double branched cover $\Sigma(J)$ of $S^3$ along $J$ as an oriented $3$-manifold and that $\Sigma(J)$ and $\Sigma(rJ)$ are isomorphic as oriented $3$-manifolds.
Therefore $\lk_{-J \sharp K}$, $\lk_{-J \sharp rK}$, $\lk_{-rJ \sharp K}$ and $\lk_{-rJ \sharp rK}$ are all isometric, and consequently the obstructions by Corollary~\ref{cor_obstruction1} and Proposition~\ref{prop_obstruction2} provide the same lower bounds of $d(J, K)$, $d(J, rK)$, $d(rJ, K)$ and $d(rJ, rK)$.

\subsection{Examples}

We have written a computer program which given Seifert surfaces of knots $J$ and $K$ calculates the linking pairing $\lk_{-J \sharp K}$ and determines whether the obstructions by Corollary~\ref{cor_obstruction1} and Proposition~\ref{prop_obstruction2} show that $d(J, K) \geq 2$ and $d(J, K) \geq 3$ respectively.
It is well-known that the linking pairing $\lk_J$ of $J$ is isometric to $\lambda(A+A^T)$ for a Seifert matrix $A$ of $J$.
We have computed the results for all pairs of nontrivial knots (possibly non-prime) together with their mirror images with up to $10$ crossings, using the database of Seifert surfaces in Knotinfo~\cite{LM}, and compared the results with the following lower bounds.
The full details of the results are available on the website \cite{FKS}.

\begin{enumerate}
\item The signature $\sigma(J)$ gives the lower bound~\cite{Muk85, Mus65}:
\[ d(J, K) \geq \frac{1}{2} \left| \sigma(J) - \sigma(K) \right| \]
\item The Rasmussen $s$-invariant $s(J)$ gives the lower bound~\cite{Ra10}:
\[ d(J, K) \geq \frac{1}{2} \left| s(J) - s(K) \right| \]
\item The Ozsv\'ath-Szab\'o $\tau$-invariant $\tau(J)$ gives the lower bound~\cite{OS03}:
\[ d(J, K) \geq \left| \tau(J) - \tau(K) \right| \]
\item For every odd prime $p$ the rank of $H_1(\Sigma(J); \F_p)$ over a finite field $\F_p$ gives the lower bound (see Corollary~\ref{cor_rank}):
\[ d(J, K) \geq \left| \rank H_1(\Sigma(J); \F_p) - \rank H_1(\Sigma(K); \F_p) \right|. \]
\item Torisu~\cite{To98} showed a necessary and sufficient condition for pairs of $2$-bridge knots of $d(J, K) = 1$ in terms of their $2$-bridge notations.
\end{enumerate}

We denote by $g_4^{smooth}(J)$ the smooth $4$-ball genus of $J$.
The lower bounds (1), (2) and (3) follow from the inequality $d(J, K) \geq g_4^{smooth}(-J \sharp K)$, the lower bounds of $g_4^{smooth}(J)$ by the invariants and the additivity of them. 
Note that $-\sigma(J) = s(J) = 2 \tau(J)$ for an alternating knot $J$.
The $s$-invariant and the $\tau$-invariant have been computed for all prime knots with up to $12$ crossings (see \cite{LM} and \cite{BG12}).
The lower bound (4) is well-known, and see Section~\ref{subsec_rank} for a proof.

An upper bound of the Gordian distance is given by the triangle inequality for a knot $I$:
\[ d(J, K) \leq d(I, J) + d(I, K). \]
In particular, when $I$ is the unknot, we have the following upper bound:
\[ d(J, K) \leq u(J) + u(K). \]
The unknotting number has been computed for all prime knots with up to $10$ crossings, except for $10_{11}$, $10_{47}$, $10_{51}$, $10_{54}$, $10_{61}$, $10_{76}$, $10_{77}$, $10_{79}$ and $10_{100}$.
(It is known that $u(K) = 2$ or $3$ for the above excepted knots $K$.)
Literature on the unknotting number is extensive.
See \cite{LM} and the references given there.
We do not have useful upper bounds on $d(J, K)$ using any other numerical invariants of knots.

As discussed in Section~\ref{subsec_obstructions} the obstructions by Corollary~\ref{cor_obstruction1} and Proposition~\ref{prop_obstruction2} provide the same lower bounds of $d(J, K)$, $d(J, rK)$, $d(rJ, K)$ and $d(rJ, rK)$.
Also, by definition, $d(J, K) = d(K, J) = d(mJ, mK) = d(mK, mJ)$, where we write $mJ$ for the mirror image of a knot $J$.
Therefore in describing numbers of pairs in our computational results we do not distinguish the pairs $(J, K)$, $(rJ, K)$, $(J, rK)$ and $(rJ, rK)$, and neither the pairs $(J, K)$, $(K, J)$, $(mJ, mK)$ and $(mK, mJ)$.

Applying Corollary~\ref{cor_obstruction1}, we have the following results.
Among all pairs $(J, K)$ of nontrivial prime knots with up to $10$ crossings such that at least one of $J$ and $K$ is not a $2$-bridge knot there are $10272$ pairs for which Corollary~\ref{cor_obstruction1} shows that $d(J, K) \geq 2$ but the value on the right hand side in any of the $4$ lower bounds (1), (2), (3) and (4) is $0$ or $1$.
For $1853$ pairs $(J, K)$ of knots among the $10272$ ones, it is known that $u(J) + u(K) \leq 2$, and hence we can conclude that $d(J, K) = u(J) + u(K) = 2$.
Also, if we count the numbers of such pairs of knots, instead, for all pairs $(J, K)$ of nontrivial knots (possibly non-prime) with up to $10$ crossings such that at least one of $J$ and $K$ is not a $2$-bridge knot, then the first one $10272$ is replaced by $12577$, and the second one $1853$ is unchanged.

\begin{exmp} \label{exmp_1}
Let $(J, K) = (8_{17}, 8_{21})$, where the bridge indices of $8_{17}$ and $8_{21}$ are equal to $3$, and $\det(8_{17}) = 37$ and $\det(8_{21}) = 15$.
A calculation of $\lk_{-8_{17} \sharp 8_{21}}$ from their Seifert surfaces shows that there exists no generator $g$ of $H_1(\Sigma(-8_{17} \sharp 8_{21}))$ such that
\[ \lk_{-8_{17} \sharp 8_{21}}(g, g) = \pm \frac{2}{37 \cdot 15} \in \Z /  (37 \cdot 15) \Z. \]
It thus follows from Corollary~\ref{cor_obstruction1} that $d(8_{17}, 8_{21}) \geq 2$.
On the other hand, since $-\sigma(8_{17}) = s(8_{17}) = 2 \tau(8_{17}) = 0$ and $-\sigma(8_{21}) = s(8_{21}) = 2 \tau(8_{21}) = 2$, we have $\frac{1}{2} |\sigma(8_{17}) - \sigma(8_{21})| = \frac{1}{2} |s(8_{17}) - s(8_{21})| = |\tau(8_{17}) - \tau(8_{21})| = 1$.
Also, since $\rank H_1(\Sigma(8_{17}); \F_p)$ is equal to $1$ for $p = 37$ and to $0$ for other odd primes $p$, and since $\rank H_1(\Sigma(8_{21}); \F_p)$ is equal to $1$ for $p = 3, 5$ and to $0$ for other $p$, it follows that $|\rank H_1(\Sigma(8_{17}); \F_p) - \rank H_1(\Sigma(8_{21}); \F_p)|$ is equal to $1$ for $p = 3, 5, 37$ and to $0$ for other $p$.
Thus any of the lower bounds (1), (2), (3) and (4) only gives $d(8_{17}, 8_{21}) \geq 1$.
Moreover, since it is known that $u(8_{17}) = u(8_{21}) = 1$, we have $d(8_{17}, 8_{21}) \leq u(8_{17}) + u(8_{21}) = 2$.
Therefore we can conclude that $d(8_{17}, 8_{21}) = 2$.
\end{exmp}

Applying Proposition~\ref{prop_obstruction2}, we have the following results.
Among all pairs $(J, K)$ of nontrivial prime knots with up to $10$ crossings there are $886$ pairs for which Proposition~\ref{prop_obstruction2} shows that $d(J, K) \geq 3$ but the value on the right hand side in any of the $4$ lower bounds (1), (2), (3) and (4) is less than $3$.
For $195$ pairs $(J, K)$ of knots among the $886$ ones, it is known that $u(J) + u(K) \leq 3$, and hence we can conclude that $d(J, K) = u(J) + u(K) = 3$.
Also, if we count the numbers of such pairs of knots, instead, for all pairs $(J, K)$ of nontrivial knots (possibly non-prime) with up to $10$ crossings, then the first one $886$ is replaced by $1696$, and the second one $195$ is by $360$.

\begin{exmp} \label{exmp_2}
Let $(J, K) = (8_7, 9_{40})$, where $\det(8_7) = 23$ and $\det(9_{40}) = 75$.
A calculation of $\lk_{-8_7 \sharp 9_{40}}$ from their Seifert surfaces show that there exists no matrix $C \in \mathcal{C}_d \cup \mathcal{C}_{-d}$, where $d = 23 \cdot 75$, such that $\lambda(C)$ is isometric to $2 \lk_{-8_7 \sharp 9_{40}}$.
It thus follows from Proposition~\ref{prop_obstruction2} that $d(8_7, 9_{40}) \geq 3$.
On the other hand, since $-\sigma(8_7) = s(8_7) = 2 \tau(8_7) = -2$ and $-\sigma(9_{40}) = s(9_{40}) = 2 \tau(9_{40}) = 2$, we have $\frac{1}{2} |\sigma(5_2) - \sigma(9_{40})| = \frac{1}{2} |s(5_2) - s(9_{40})| = |\tau(5_2) - \tau(9_{40})| = 2$.
Note that both of $8_7$ and $9_{40}$ are alternating knots.
Also, since $\rank H_1(\Sigma(8_7); \F_p)$ is equal to $1$ for $p = 23$ and to $0$ for other odd primes $p$, and since $\rank H_1(\Sigma(9_{40}); \F_p)$ is equal to $2$ for $p = 5$, to $1$ for $p = 3$ and to $0$ for other $p$, it follows that $|\rank H_1(\Sigma(8_7); \F_p) - \rank H_1(\Sigma(9_{40}); \F_p)|$ is equal to $2$ for $p = 5$, to $1$ for $p = 3, 23$ and to $0$ for other $p$.
Thus any of the lower bounds (1), (2), (3) and (4) only gives $d(8_7, 9_{40}) \geq 2$.
Moreover, since it is known that $u(8_7) = 1$ and $u(9_{40}) = 2$, we have $d(8_7, 9_{40}) \leq u(8_7) + u(9_{40}) = 3$.
Therefore we can conclude that $d(8_7, 9_{40}) = 3$.
\end{exmp}

\begin{exmp} \label{exmp_3}
Let $(J, K) = (3_1, 4_1 \sharp 4_1)$, where $\det(3_1) = 3$ and $\det(4_1 \sharp 4_1) = 25$.
Similarly, a calculation of $\lk_{-3_1 \sharp 4_1 \sharp 4_1}$ shows that we can apply Proposition~\ref{prop_obstruction2} to have $d(3_1, 4_1 \sharp 4_1) \geq 3$.
Since $-\sigma(3_1) = s(3_1) = 2 \tau(3_1) = -2$ and $-\sigma(4_1 \sharp 4_1) = s(4_1 \sharp 4_1) = 2 \tau(4_1 \sharp 4_1) = 0$, we have $\frac{1}{2} |\sigma(3_1) - \sigma(4_1 \sharp 4_1)| = \frac{1}{2} |s(3_1) - s(4_1 \sharp 4_1)| = |\tau(3_1) - \tau(4_1 \sharp 4_1)| = 1$.
Note that both of $3_1$ and $4_1 \sharp 4_1$ are alternating knots, and that $4_1 \sharp 4_1$ is a slice knot.
Also, since $\rank H_1(\Sigma(3_1); \F_p)$ is equal to $1$ for $p = 3$ and to $0$ for other odd primes $p$, and since $\rank H_1(\Sigma(4_1 \sharp 4_1); \F_p)$ is equal to $2$ for $p = 5$ and to $0$ for other $p$, it follows that $|\rank H_1(\Sigma(3_1); \F_p) - \rank H_1(\Sigma(4_1 \sharp 4_1); \F_p)|$ is equal to $2$ for $p = 5$, to $1$ for $p = 3$ and to $0$ for other $p$.
Thus any of the lower bounds (1), (2) and (3) only gives $d(3_1, 4_1 \sharp 4_1) \geq 1$, and the lower bound (4) does $d(3_1, 4_1 \sharp 4_1) \geq 2$.
Moreover, since it is known that $u(3_1) = u(4_1) = 1$, we have $d(3_1, 4_1 \sharp 4_1) \leq u(3_1) + u(4_1 \sharp 4_1) \leq u(3_1) + 2 u(4_1) = 3$.
Therefore we can conclude that $d(3_1, 4_1 \sharp 4_1) = 3$.
\end{exmp}

\begin{rem}
\begin{enumerate}
\item Kawauchi~\cite{K12}, Murakami~\cite{Muk85} and Nakanishi~\cite{N81} gave obstructions for $d(J, K) = 1$ in terms of the Alexander polynomial.
\item Darcy and Summers~\cite{DS97, DS98}, Motegi~\cite{Mot96} and Torisu~\cite{To98} gave obstructions for $d(J, K) = 1$ applicable for many pairs of $2$-bridge knots or Montesinos knots.
\item Darcy~\cite{D} and Moon~\cite{Moo10} gave tables of lower and upper bounds of the Gordian distance of many pairs of knots with $10$ crossings or less.
For example, their tables describe $1 \leq d(8_{17}, 8_{21}) \leq 2$ and $2 \leq d(3_1, 4_1 \sharp 4_1) \leq 3$, but our lower bound determines that $d(8_{17}, 8_{21}) = 2$ and $d(3_1, 4_1 \sharp 4_1) = 3$ as in Examples~\ref{exmp_1} and \ref{exmp_3}.
\item Miyazawa~\cite{Mi11} showed lower bounds of the Gordian distance using the Jones, HOMFLY and $Q$-polynomials of knots and gave lists of the Gordian distance or its lower bounds for many pairs of knots with $10$ crossings or less.
For example, Miyazawa~\cite[Corollary 3.5]{Mi11} showed that $d(8_7, 9_{40}) \geq 2$, but our lower bound determines that $d(8_7, 9_{40}) = 3$ as in Example~\ref{exmp_2}.
\end{enumerate}
\end{rem}

\subsection{Minimal number of generators of the Alexander modules} \label{subsec_rank}

For readers' convenience we prove the following well-known lower bounds of the Gordian distance.
Wendt~\cite{W37} first gave lower bounds by the ranks of the first homology groups of the cyclic branched covers.
For a knot $J$ and a homomorphism $\varphi \colon \Lambda \to R$, which is not necessarily admissible, we denote by $m_\varphi(J)$ the minimal number of generators of $H_1^\varphi(X_J; R)$.

\begin{prop}
Let $J$ and $K$ be knots in $S^3$, and let $\varphi \colon \Lambda \to R$ be a homomorphism over a principal ideal domain.
Then $d(J, K) \geq |m_\varphi(J) - m_\varphi(K)|$.
\end{prop}

This proposition can be easily deduced from the following lemma.

\begin{lem}
Let $J$ and $K$ be knots in $S^3$ such that $J$ can be turned into $K$ by a single crossing change, and let $\varphi \colon \Lambda \to R$ be an admissible homomorphism over a principle ideal domain.
Then $|m_\varphi(J) - m_\varphi(K)|$ is equal to $0$ or $1$. 
\end{lem}

\begin{proof}
Let $c$ be a simple closed curve in $X_J$ corresponding to a crossing change turning $J$ into $K$ as in the proof of Proposition~\ref{prop_construction}.
Note that $c$ is nullhomologous.
Let $Z$ be an open tubular neighborhood of $c$ in $X_J$, and we denote by $Y$ the complement of $Z$.
Then we consider the Mayer-Vietoris homology exact sequence:
\[ H_1(\partial Z; R) \to H_1^\varphi(Y; R) \oplus H_1(Z; R) \to H_1^\varphi(X_J; R) \to 0, \]
where $H_1(Z; R)$ and $H_1(\partial Z; R)$ are free $R$-modules of rank $1$ and $2$ respectively, and the inclusion induced homomorphism $H_1(\partial Z; R) \to H_1(Z; R)$ is surjective.
Now by a standard argument on finitely generated modules over a principal ideal domain we have $m - 1 \leq m_\varphi(J) \leq m$, where $m$ is the minimal number of generators of $H_1^\varphi(Y; R)$.
Similarly, we also have $m - 1 \leq m_\varphi(K) \leq m$.
Therefore we have $|m_\varphi(J) - m_\varphi(K)| \leq 1$.
\end{proof}

In the case of the admissible homomorphism $\varphi_{-1} \colon \Lambda \to \F_p$ sending $t$ to $-1$ for an odd prime $p$, the twisted homology group $H_1^{\varphi_{-1}}(X_J; \F_p)$ is isomorphic to $H_1(\Sigma(J); \F_p)$.
(See for instance \cite[Lemma 3.3]{BF15}.)
Thus we have the following corollary.

\begin{cor} \label{cor_rank}
Let $J$ and $K$ be knots in $S^3$, and let $p$ be an odd prime.
Then
\[ d(J, K) \geq |\rank H_1(\Sigma(J); \F_p) - \rank H_1(\Sigma(K); \F_p)|. \] 
\end{cor}

%%%%%%% References %%%%%%%%%%%%%%%%%%%%%%%%%%%%%%%%%%%%%%%%%%%%%%%%%%%%%%%%%%%%


\begin{thebibliography}{FKLMN20}
\bibitem[BG12]{BG12} J.~A.~Baldwin and W.~D.~Gillam,
\textit{Computations of Heegaard-Floer knot homology},
J.\ Knot Theory Ramifications \textbf{21} (2012), no. 8, 1250075, 65 pp.

\bibitem[BF]{BF} M.~Borodzik and S.~Friedl,
\textit{Knotrious world wide web page},
available at https://www.mimuw.edu.pl/~mcboro/knotorious.php

\bibitem[BF14a]{BF14a} M.~Borodzik and S.~Friedl,
\textit{The unknotting number and classical invariants II},
Glasg.\ Math.\ J.\ \textbf{56} (2014), no. 3, 657--680.

\bibitem[BF14b]{BF14b} M.~Borodzik and S.~Friedl,
\textit{On the algebraic unknotting number},
Trans.\ London Math.\ Soc.\ \textbf{1} (2014), no. 1, 57--84.

\bibitem[BF15]{BF15} M.~Borodzik and S.~Friedl,
\textit{The unknotting number and classical invariants I},
Algebr.\ Geom.\ Topol.\ \textbf{15} (2015), no. 1, 85--135.

\bibitem[BFP16]{BFP16} M.~Borodzik, S.~Friedl and M.~Powell,
\textit{Blanchfield forms and Gordian distance},
J.\ Math.\ Soc.\ Japan \textbf{68} (2016), no. 3, 1047--1080.

\bibitem[C19]{C19} J.~Chen,
\textit{Algebraic Gordian distance}, 
J.\ Knot Theory Ramifications \textbf{28} (2019), no. 4, 1950024, 13 pp. 

\bibitem[CoLi86]{CoLi86} T.~D.~Cochran and W.~B.~R.~Lickorish,
\textit{Unknotting information from $4$-manifolds},
Trans.\ Amer.\ Math.\ Soc.\ \textit{297} (1986), no. 1, 125--142.

\bibitem[CS99]{CS99} J.~H.~Conway and N.~J.~A.~Sloane,
\textit{Sphere packings, lattices and groups},
Third edition, With additional contributions by E.~Bannai, R.~E.~Borcherds, J.~Leech, S.~P.~Norton, A.~M.~Odlyzko, R.~A.~Parker, L.~Queen and B.~B.~Venkov, Grundlehren der mathematischen Wissenschaften [Fundamental Principles of Mathematical Sciences], 290. Springer-Verlag, New York, 1999.

\bibitem[D]{D} I.~K.~Darcy,
\textit{Strand passage metric table},
available at \url{http://www.math.uiowa.edu/˜idarcy/TAB/tabnov.pdf}.

\bibitem[DS97]{DS97} I.~K.~Darcy and D.~W.~Sumners,
textit{A strand passage metric for topoisomerase action},
KNOTS '96 (Tokyo), 267–278, World Sci. Publ., River Edge, NJ, 1997.

\bibitem[DS98]{DS98} I.~K.~Darcy and D.~W.~Sumners,
\textit{Applications of topology to DNA},
Knot theory (Warsaw, 1995), 65--75, Banach Center Publ., 42, Polish Acad.\ Sci.\ Inst.\ Math., Warsaw, 1998.

\bibitem[F93]{F93} M.~E.~Fogel, 
\textit{The algebraic unknotting number},
Thesis (Ph.D.)–University of California, Berkeley, 1993, 50 pp.

\bibitem[FKLMN20]{FKLMN20} S.~Friedl, T.~Kitayama, L.~Lewark, M.~Nagel and M.~Powell,
\textit{Seifert forms, twisted Alexander polynomials and homotopy ribbon concordance},
Canad.\ J.\ Math.\ \textbf{74} (2022), no. 4, 1137--1176.

\bibitem[FKS]{FKS} S.~Friedl, T.~Kitayama and M.~Suzuki,
\textit{Lower bound of the Gordian distance},
available at \url{https://www.isc.meiji.ac.jp/~mackysuzuki/lowerboundofgordiandistance.html}.

\bibitem[H12]{H12} J.~Hillman,
\textit{Algebraic invariants of links},
Second edition, Series on Knots and Everything, 52, World Scientific Publishing Co. Pte. Ltd., Hackensack, NJ, 2012. xiv+353 pp.

\bibitem[HS71]{HS71} P.~J.~Hilton and U.~Stammbach,
\textit{A course in homological algebra},
Graduate Texts in Mathematics, Vol.\ 4, Springer-Verlag, New York-Berlin, 1971, ix+338 pp.

\bibitem[J09]{J09} S.~Jabuka,
\textit{The rational Witt class and the unknotting number of a knot},
arXiv:0907.2275.

\bibitem[K12]{K12} A.~Kawauchi,
\textit{On the Alexander polynomials of knots with Gordian distance one},
Topology Appl.\ \textbf{159} (2012), no. 4, 948-–958.

\bibitem[Le69]{Le69} J.~Levine,
\textit{Invariants of knot cobordism},
Invent.\ Math.\ \textbf{8} (1969), 98–-110.

\bibitem[Le77]{Le77} J.~Levine,
\textit{Knot modules. I},
Trans.\ Amer.\ Math.\ Soc.\ \textbf{229} (1977), 1--50.

\bibitem[Lic85]{Lic85} W.~B.~R.~Lickorish,
\textit{The unknotting number of a classical knot},
Combinatorial methods in topology and algebraic geometry (Rochester, N.Y., 1982), 117--121, Contemp.\ Math., \textbf{44}, Amer.\ Math.\ Soc., Providence, RI, 1985.

\bibitem[Lic97]{Lic97} W.~B.~R.~Lickorish,
\textit{An introduction to knot theory}, 
Graduate Texts in Mathematics, 175, Springer-Verlag, New York, 1997, x+201 pp.

\bibitem[Liv20]{Liv20} C.~Livingston,
\textit{Signature invariants related to the unknotting number},
Pacific J.\ Math.\ \textit{305} (2020), no. 1, 229--250.

\bibitem[LM]{LM} C.~Livingston and A.~H.~Moore,
\textit{KnotInfo: Table of Knot Invariants},
\url{knotinfo.math.indiana.edu}, August, 2022.

\bibitem[Mi11]{Mi11} Y.~Miyazawa,
\textit{Gordian distance and polynomial invariants},
J.\ Knot Theory Ramifications \textbf{20} (2011), no. 6, 895--907.

\bibitem[Moo10]{Moo10} H.~Moon,
\textit{Calculating knot distances and solving tangle equations involving montesinos links}, Thesis (Ph.D.)–The University of Iowa, 2010, 154 pp.

\bibitem[Mot96]{Mot96} K.~Motegi,
\textit{Bridge numbers of twisted Montesinos knots},
Proceedings of The Institute of Natural Sciences, Nihon University No. 31 (1996).

\bibitem[Muk85]{Muk85} H.~Murakami,
\textit{Some metrics on classical knots},
Math.\ Ann.\ \textbf{270} (1985), no. 1, 35--45.

\bibitem[Muk90]{Muk90} H.~Murakami,
\textit{Algebraic unknotting operation},
Proceedings of the Second Soviet-Japan Joint Symposium of Topology (Khabarovsk, 1989), Questions Answers Gen.\ Topology \textbf{8} (1990), no. 1, 283--292.

\bibitem[Mus65]{Mus65} K.~Murasugi,
\textit{On a certain numerical invariant of link types},
Trans.\ Amer.\ Math.\ Soc.\ \textbf{117} (1965), 387–-422.

\bibitem[N81]{N81} Y.~Nakanishi,
\textit{A note on unknotting number},
Math.\ Sem.\ Notes Kobe Univ.\ \textbf{9} (1981), no. 1, 99--108.

\bibitem[OS03]{OS03} P.~Ozsv\'ath and Z.~Szab\'o,
\textit{Knot Floer homology and the four-ball genus},
Geom.\ Topol.\ \textbf{7} (2003), 615--639.

\bibitem[Ra10]{Ra10} J.~Rasmussen,
\textit{Khovanov homology and the slice genus},
Invent.\ Math.\ \textbf{182} (2010), no. 2, 419--447.

\bibitem[Ro09]{Ro09} J.~J.~Rotman,
\textit{An introduction to homological algebra},
Second edition, Universitext, Springer, New York, 2009. xiv+709 pp.

\bibitem[Ta69]{Ta69} L.~R.~Taylor,
\textit{On the genera of knots},
Topology of low-dimensional manifolds (Proc. Second Sussex Conf., Chelwood Gate, 1977), pp. 144–-154, Lecture Notes in Math., 722, Springer, Berlin, 1979.

\bibitem[To98]{To98} I.~Torisu,
\textit{The determination of the pairs of two-bridge knots or links with Gordian distance one}, Proc.\ Amer.\ Math.\ Soc.\ \textbf{126} (1998), no. 5, 1565--1571.

\bibitem[Tr69]{Tr69} A.~G.~Tristram,
\textit{Some cobordism invariants for links}, 
Proc.\ Cambridge Philos.\ Soc.\ \textit{66} (1969), 251--264.

\bibitem[Tu01]{Tu01} V.~Turaev,
\textit{Introduction to combinatorial torsions}, Notes taken by Felix Schlenk, Lectures in Mathematics ETH Z\"urich. Birkh\"auser Verlag, Basel, 2001. viii+123 pp.

\bibitem[W37]{W37} H.~Wendt,
\textit{Die gordische Auflösung von Knoten},
Math.\ Z.\ \textbf{42} (1937), no. 1, 680--696.
\end{thebibliography}
\end{document}